\documentclass{article}
\usepackage{amscd, amsbsy, amsfonts}
\usepackage{amsmath, amssymb, amsthm, color}

\DeclareMathOperator{\Irr}{Irr}     \DeclareMathOperator{\IBr}{IBr}

\DeclareMathOperator{\GL}{GL}

     \DeclareMathOperator{\SL}{SL}

\newtheorem{thm}{Theorem}[section]

\newtheorem{lem}[thm]{Lemma}

\newtheorem{prop}[thm]{Proposition}

\begin{document}
\title{Groups with one or two super-Brauer character theories}
\author{{\rm Xiaoyou Chen$^{1}$}{\thanks{
{\it E-mail address}: cxymathematics$@$hotmail.com}},
Mark L. Lewis$^{2}${\thanks{Corresponding author
{\it E-mail address}: lewis$@$math.kent.edu}}\\
{\small({\small {\rm 1.} College of Science, Henan University of Technology, Zhengzhou 450001, China;}
}\\
{\small{\small {\rm 2.} Department of Mathematical Sciences, Kent State University, Kent 44242, USA)}}
}

\date{}
\maketitle
%%%-------------------------------------------------------------------------------------------
%{\bf \hrule} \vspace{3mm}
\begin{abstract}
A super-Brauer character theory of a group $G$ and a prime $p$ is a pair consisting of a partition of the irreducible $p$-Brauer characters and
a partition of the $p$-regular elements of $G$ that satisfy certain properties.
We classify the groups and primes that have exactly one super-Brauer character theory.
We discuss the groups with exactly two super-Brauer character theories.
\end{abstract}

\vspace{3mm}
\noindent{\bf Key words:} super character theory; super-Brauer character theory

\vspace{3mm}
\noindent{\bf 2010 Mathematics Subject Classification:} 20C20
%\vspace{3mm} \hrule

%%%-------------------------------------------------------------------------------------------

\section{Introduction}

%\hspace{1.2em}
In this paper, all groups are finite.  In {\cite{Ref4}}, Diaconis and Isaacs defined supercharacter theories for an arbitrary group, and developed some properties of supercharacter theories.  Roughly speaking, in a supercharacter theory, the irreducible characters are replaced by supercharacters and the conjugacy classes are replaced by super-classes.  We will discuss supercharacter theories in more detail in the next section.  Supercharacter theories have been applied to study several kinds of groups, and some examples can be seen in {\cite{Ref1}}, {\cite{Ref2}}, {\cite{Ref3}} and {\cite{Ref4}}.

For a fixed prime $p$, the $p$-Brauer characters of $G$ play a similar role for the characteristic $p$ representations of $G$ that ordinary characters paly for the characteristic $0$ representations of $G$.  One theme in Brauer character theory is to determine what properties of the ordinary characters can be translated to properties of Brauer characters.  With this in mind, the first author and Zeng defined an analogous ``super theory'' for $p$-Brauer characters in \cite{Chen2011}.

It is not difficult to see that the trivial group and $\mathbb{Z}_2$ are the only groups that have only one supercharacter theory.  In \cite{Lewis201506}, Burkett, Lamar, the second author, and Wynn classified groups with exactly two supercharacter theories.  They proved that a group $G$ has exactly two supercharacter theories if and only if $G$ is isomorphic to one of $\mathbb{Z}_3$, $S_3$, or ${\rm Sp}(6,2)$.  Inspired by this, we consider groups with one or two super-Brauer character theories in this note.

We begin by considering for which groups $G$ and primes $p$ does $G$ have exactly one super-Brauer character theory.
It is not difficult to see that if $G$ has only one or two $p$-regular conjugacy classes then $G$ will have only one super-Brauer character theory.  We show that these are the only possibilities.

\begin{thm}
Let $G$ be a group and let $p$ be a prime.  Then $G$ has only one super-Brauer character theory if and only if $G$ has only one or two $p$-regular conjugacy classes.
\end{thm}

Note that the identity is a $p$-regular element.  Thus, a group $G$ has one $p$-regular conjugacy class if and only if $G$ has no elements of $p'$ order.  I.e., $G$ has one $p$-regular conjugacy class if and only if $G$ is a $p$-group.  The groups with two $p$-regular conjugacy classes have been classified by Ninomiya in \cite{Ninomiya199111}.  We will provide this classification in Section \ref{sect one}.  In particular, we can use this classification to describe all groups and primes with one super-Brauer character theory.  Notice that unlike the ordinary case where there were only two such groups, in this case, there are infinitely groups that occur.

Next, we consider the groups and primes that have two super-Brauer character theories.  We will see that if a group $G$ has three $p$-regular classes, then $G$ must have two super-Brauer character theories.  We will also see that there exist groups with more than three $p$-regular classes that only have two super-Brauer character theories.  Unfortunately, with the tools we currently have available, the groups with two super-Brauer character theories that are not $p$-solvable appear to be beyond our capabilities.  In particular, there is a connection between supercharacter theories and Schur rings that is expressed in a bijection between the supercharacter theories of a group and the $S$-rings over $G$ contained in ${\bf Z} (\mathbb{C} [G])$ (see Proposition 2.4 of \cite{hen}).  This connection was exploited computationally in \cite{Lewis201506} to prove that ${\rm Sp} (6,2)$ has only two supercharacter theories.  At this time, we do not know of any such connections for super-Brauer character theory. Thus, we focus on the $p$-solvable groups that have two super-Brauer character theories.

\begin{thm}\label{two theories}
Let $G$ be a $p$-solvable group with ${\bf O}_p (G) = 1$.  Then $G$ has exactly two super-Brauer character theories if and only if either {\rm (1)} $G$ has three $p$-regular classes, or {\rm (2)} $G$ has a normal $p$-complement $M$, the subgroup $M$ is minimal normal in $G$, and $M$ has exactly two $P$-invariant supercharacter theories where $P$ is a Sylow $p$-subgroup of $G$.
\end{thm}

In a series of papers \cite{Ninomiya199102, Ninomiya199103, Ninomiya199301}, Ninomiya has classified the finite groups that have exactly three $p$-regular conjugacy classes.  Thus, we need only determine the groups that have a normal $p$-complement with exactly two supercharacter theories that are invariant under the action of a Sylow $p$-subgroup and the $p$-complement is not the union of three conjugacy classes.  Unfortunately, when we encounter groups where the normal $p$-complement is not solvable, we have the same difficulties as above; so these also seem beyond our currently capabilities.  At this time, we are able to determine these groups that arise in the solvable case.  We will do this in Theorem \ref{three P-invariant}.  Thus, we will be able to list all of the solvable groups that have exactly two super-Brauer character theories.

\section{Background}

We begin by recalling the definition of a supercharacter theory.  Let $G$ be a group and write $\Irr (G)$ for the set of irreducible characters of $G$, and partition $G$ and $\Irr (G)$ into the collections $\mathfrak{L}$ and $\mathfrak{Y}$ of nonempty subsets, respectively.  Following Diaconis and Isaacs, we say that $(\mathfrak {L}, \mathfrak {Y})$ is a {\it supercharacter theory} for $G$ if (a) $|\mathfrak {L}| = |\mathfrak {Y}|$, (b) for each set $Y \in \mathfrak{Y}$ there is a character $\chi_{Y}$ whose irreducible constituents all lie in $Y$ so that $\chi_{Y}$ is constant on the members of $\mathfrak{L}$, and (c) the set $\{ 1 \} \in \mathfrak {L}$. The characters $\chi_{Y}$ are called {\it supercharacters} and the elements of $\mathfrak{L}$ are called {\it superclasses}.  We use ${\rm Sup}(G)$ to denote the set of supercharacter theories of $G$.

Next, we recall the definition of a super-Brauer character theory.  Assume that $G$ is a group and $p$ is a fixed prime number.  Let $G^\circ$ be the set of $p$-regular elements of $G$.  Denote by $\IBr (G)$ the set of irreducible $p$-Brauer characters of $G$.  For convenience, when the prime $p$ is clear, we denote $p$-regular classes and $p$-Brauer characters by regular classes and Brauer characters, respectively.  For other terminologies and notations, one can refer to {\cite{Ref5}} and {\cite{Ref6}}.

Following \cite{Chen2011}, a {\it super-Brauer character theory} of $G$ is a pair $(\mathfrak{X}, \mathfrak{K})$ where $\mathfrak{X}$ is a partition of $\IBr (G)$ and $\mathfrak{K}$ is a partition of $G^\circ$ that satisfy (1) $|\mathfrak {X}| = |\mathfrak {K}|$, (2) for each set $X \in \mathfrak{X}$ there is a nonzero Brauer character $\theta_{X}$ whose irreducible constituents all lie in the set $X$ such that $\theta_{X}$ is constant on every $K \in \mathfrak{K}$, and (3) $\{ 1 \} \in \mathfrak {K}$.  In this case, we refer to the $\theta_{X}$'s as super-Brauer characters and the sets $K\in \mathfrak{K}$ as super-regular classes.  We write ${\rm Sup} (G^\circ)$ to denote the set of all super-Brauer character theories of $G$.

For any group $G$, there exist two {\it trivial} super-Brauer character theories:

\begin{enumerate}
\item $\mathfrak{m} (G^\circ)$ whose partitions are $\mathfrak{X}$ consists just of the singleton sets $\{ \phi \}$ where $\phi$ runs through $\IBr (G)$ and $\mathfrak{K}$ consists of the conjugacy classes in $G^\circ$

\item $\mathfrak{M} (G^\circ)$ whose partitions are $\mathfrak{X} = \{ \{ 1_{G^\circ} \}, \IBr (G) - \{ 1_{G^\circ} \} \}$ and $\mathfrak {K} = \{ \{ 1 \}, G^\circ - \{ 1 \} \}$.
\end{enumerate}

Note that $1_{G^\circ}$ is the principal Brauer character of $G$.  The super-Brauer characters in (2) are $\Phi_{1_{G^\circ}} (1) 1_{G^\circ}$ and $\rho_G^\circ - \Phi_{1_{G^\circ}} (1) 1_{G^\circ}$, where $\rho_G^\circ$ is the restriction of the regular character $\rho_G$ of $G$ to $G^\circ$ and $\Phi_{1_{G^\circ}}$ is the projective indecomposable character with respect to $1_{G^\circ}$.

\section{One super-Brauer character theory} \label{sect one}

We first classify groups with one super-Brauer character theory.  These are the groups where $\mathfrak{m} (G^\circ) = \mathfrak{M} (G^\circ)$.   Note that $p$-groups are trivially $\{ p, q \}$-groups for every prime $q$, so in the second conclusion of the following proposition, we are including $p$-groups.

\begin{prop} \label{two one}
Let $G$ be a nontrivial group and $p$ a fixed prime number.  Then $|{\rm Sup} (G^\circ)| = 1$ if and only if $G$ has at most two $p$-regular conjugacy classes.  Furthermore, if $G$ is a group with at most two $p$-regular conjugacy classes, then $G$ is a $\{ p, q \}$-group for some prime $q$.  In particular, $G$ is solvable.
\end{prop}

\begin{proof}
We see that $G$ has only one super-Brauer character theory if and only if $\mathfrak{M} (G^\circ)$ and $\mathfrak{m} (G^\circ)$ are the same.  These two super-Brauer character theories are the same if and only if $G^\circ - \{ 1 \}$ is empty or consists a single conjugacy class.  If $G^\circ-\{1\}$ is empty, then $G$ has no nontrivial $p$-regular elements, so $G$ is a $p$-group, and thus, $G$ is a $\{ p,q \}$ group for any prime $q$.  On the other hand, if $G^\circ - \{ 1 \}$ consists of a single conjugacy class, then all of the nontrivial $p$-regular elements of $G$ must have the same order.  Since $G$ has nonidentity $p$-regular elements, there must exist a prime $q \ne p$ so that $q$ divides $|G|$.  By Cauchy's theorem, we know that $G$ must have an element of order $q$, and so, all the elements of $G^\circ$ have order $q$. Thus, $G$ is a $\{ p, q \}$-group, and by Burnside's $p^{a}q^{b}$-theorem, $G$ is solvable.
\end{proof}

The groups with one nontrivial $p$-regular conjugacy class are classified by Ninomiya and Wada \cite[Theorem 3.1]{Ninomiya199111}.  Recall that $G$ and $G/{\bf O}_p (G)$ have the same number of irreducible Brauer characters, so they have the same number of $p$-regular conjugacy classes. Obviously, $G$ has one $p$-regular conjugacy class if and only if $G$ is a $p$-group, and $G$ will have two $p$-regular conjugacy classes if and only if $G/{\bf O}_p(G)$ is one of the groups listed in the conclusion of \cite[Theorem 3.1]{Ninomiya199111} (or \cite[Theorem A]{Ninomiya199103}).  We have listed these groups in Table \ref{two classes}.  In light of Proposition \ref{two one}, $G$ will have only one super-Brauer character theory if and only if $G = {\bf O}_p (G)$ or $G/{\bf O}_p (G)$ is one of the groups listed in Table \ref{two classes}.
%In particular, let $G$ be a $p$-solvable group with ${\bf O}_{p} (G) = 1$. Then $G$ has exactly two $p$-regular classes if and only if one of the following holds: (1) $p$ is odd and $G \simeq \mathbb{Z}_{2}$, (2) $p=2$ and $G\simeq E_{3^{2}}\rtimes P$ where $P$ is one of $\mathbb{Z}_{8}$, $Q_{8}$, or $S_{16}$, where $S_{16}$ is the semi-dihedral group of order $16$, (3) $p=2$ and $G\simeq \mathbb{Z}_{q}\rtimes \mathbb{Z}_{2^{n}}$, where $q=2^{n}+1$ is a Fermat prime, or (4) $p=2^{n}-1$ (a Mersenne prime) and $G\simeq E_{2^{n}}\rtimes \mathbb{Z}_{p}$.

\begin{table}
$$\begin{tabular}{||c|c||}
\hline
\hline
$p$ & $G/{\bf O}_p (G)$ \\
\hline
\hline
any prime & $1$ \\
\hline
any odd prime & $\mathbb{Z}_{2}$ \\
\hline
$2$ & $E_{3^{2}}\rtimes P$ where $P$ is one of $\mathbb{Z}_{8}$, $Q_{8}$, or $S_{16}$ \\
 & $Q_8$ is the quaternions of order $8$ \\
 & $S_{16}$ is the semi-dihedral group of order $8$ \\
\hline
$2$ &  $\mathbb{Z}_{q}\rtimes \mathbb{Z}_{2^{n}}$, where $q=2^{n}+1$ is a Fermat prime \\
\hline
$2^{n}-1$ (a Mersenne prime) & $G\simeq E_{2^{n}}\rtimes \mathbb{Z}_{p}$ \\
\hline
\hline
\end{tabular}$$
\caption{Groups $G$ which have at most two $p$-regular classes}\label{two classes}
\end{table}

\section{$A$-invariant supercharacter theories}

In general, it is not clear that there is any relationship between supercharacter theories and super-Brauer character theories.  We will now show that in one case there is a relationship.

Let $G$ be a group and let $A$ be a group that acts on $G$ via automorphisms.  We can define an action of $A$ on $\Irr (G)$ by $\chi^a (g^a) = \chi (g)$ for all $a \in A$, $g \in G$ and $\chi \in \Irr (G)$.  This defines an action on subsets and partitions of $G$  and $\Irr (G)$.  It is not difficult to see that $A$ will act on supercharacter theories of $G$.  In addition, a supercharacter theory of $G$ will be $A$-invariant if and only if each set in the partition of $G$ is a union of $A$-orbits of $G$ and each set in the partition of $\Irr (G)$ will be a union of $A$-orbits of $\Irr (G)$.  We write ${\rm Sup}_A (G)$ for the set of $A$-invariant supercharacter theories of $G$.

\begin{thm} \label{Hall}
Let $N$ be a normal Hall $p$-complement of a group $G$, and let $P$ be a Sylow $p$-subgroup of $G$.  Then there is a bijection between ${\rm Sup} (G^\circ)$ and ${\rm Sup}_P (N)$.  Furthermore, the corresponding theories will have the same partition of $G^\circ = N$.
\end{thm}

\begin{proof}
Since $N$ is a $p'$-group, we know that $\Irr (N) = \IBr (N)$.  In view of Green's theorem (Theorem 8.11 of \cite{Ref6}), we know that for every character $\theta \in \Irr (N)$ there is a unique Brauer character $\phi \in \IBr (G)$ that lies over $\theta$ and $\phi_N = \sum_{i=1}^n \theta_i$ where $\{ \theta_1 = \theta, \dots, \theta_n \}$ is the $P$-orbit of $\theta$.  This gives a bijection from $\IBr (G)$ to the $P$-orbits of $\Irr (N)$.  Given a Brauer character $\phi \in \IBr (G)$, we write $O (\phi)$ for the $P$-orbit of $\Irr (N)$ that consists of the irreducible constituents of $\phi_N$.

If $X \subseteq \IBr (G)$, then we write $\bar X = \cup_{\phi \in X} O (\phi)$.  If $\mathfrak{X}$ is a partition of $\IBr (G)$, then $\bar {\mathfrak {X}} = \{ \bar X \mid X \in \mathfrak {X} \}$ will be a $P$-invariant partition of $\Irr (N)$.  Notice that $N = G^\circ$, and so, any partition of $N$ will be a partition of $G^\circ$ and vice-versa. In particular, any partition of $G^\circ$ where the sets are unions of $G$-conjugacy classes will have the sets are unions of $P$-orbits and so, will be a $P$-invariant partition of $N$.

Suppose $(\mathfrak{X},\mathfrak{K})$ is super-Brauer character theory for $G$.  We claim that $(\bar{\mathfrak{X}},\bar {\mathfrak{K}})$ is a $P$-invariant supercharacter theory for $N$.
We know for each set $X \in \mathfrak {X}$ that there exist integers $a_\phi$ so that $\sum_{\phi \in X} a_\phi \phi$ is constant on $K$ for all $K \in \mathfrak{K}$.  Observe that if $K \in \mathfrak{K}$ and $k \in K$, then $k \in N$. It follows for $X \in \mathfrak{X}$ that
$$\sum_{\theta \in \bar X} a_\theta \theta (k) = \sum_{\phi \in X} \sum_{\theta \in O(\phi)} a_\theta \theta (k) = \sum_{\phi \in X} a_\phi \phi (k)$$
where $a_\theta = a_\phi$ for $\theta \in O(\phi)$.  This implies that $\sum_{\theta \in \bar X} a_\theta \theta$ is constant on each $K \in \mathfrak{K}$.  It also implies that $(\bar {\mathfrak {X}},\mathfrak {K})$ is a $P$-invariant supercharacter theory for $N$.  Thus, the map $(\mathfrak{X},\mathfrak{K}) \mapsto (\bar{\mathfrak {X}},\mathfrak{K})$ is a well-defined function from ${\rm Sup} (G^\circ)$ to ${\rm Sup}_P (N)$.

Note that different super-Brauer character theories must have different partitions of $G^\circ$, so this map will be an injection.  Given a supercharacter theory $(\mathfrak {Y},\mathfrak {K}) \in {\rm Sup}_P (N)$, for $Y \in \mathfrak {Y}$, define $X (Y)$ by $X (Y)  = \{ \phi \mid O(\phi) \subseteq Y \}$ and set $\mathfrak {X} = \{ X(Y) \mid Y \in \mathfrak {Y} \}$.  It is not difficult to see $\bar {\mathfrak {X}} = \mathfrak {Y}$ and $(\mathfrak {X},\mathfrak {K})$ is a super-Brauer character theory.  This yields the stated bijection.
\end{proof}

\section{Two super-Brauer character theories}

Now, we turn our attention to groups with exactly two super-Brauer character theories.  For any group with at least three $p$-regular conjugacy classes, the two super-Brauer character theories $\mathfrak{M} (G^\circ)$ and $\mathfrak{m} (G^\circ)$ will be distinct.  It is easy to see that if $G$ has only three $p$-regular classes,
then $\mathfrak{m}(G^\circ)$ and $\mathfrak{M}(G^\circ)$ are the only possible super-Brauer character theories.
The interesting question is whether there exist groups with more than three $p$-regular classes that have only two super-Brauer theories.

We will present some solvable examples where this occurs later.  The only nonsolvable example that we know of where this occurs is $G = {\rm SP}_6 (2) \times P$ where $P$ is a $p$-group for some prime $p$ that does not divide $|{\rm SP}_6 (2)|$, but we would be surprised if other examples do not exist.

Determining the nonsolvable groups with two super-Brauer character theories seems out of reach at this time.
At this point, we consider only the $p$-solvable case.

\begin{proof}[Proof of Theorem \ref{two theories}]
Assume that $G$ has exactly two super-Brauer character theories: $\mathfrak{m}(G^\circ)$ and $\mathfrak{M}(G^\circ)$.  Let $M$ be a minimal normal subgroup of $G$.  Since $G$ is $p$-solvable and ${\bf O}_p(G) = 1$, it follows that $p$ does not divide $|M|$, so $M\subseteq G^\circ$.

Suppose that $M < G^\circ$.  Let $\mathfrak{R}$ be the partition of $G^\circ$ by $\{ \{ 1 \}, M - \{ 1 \}, G^\circ - M \}$, and let $\mathfrak{S}$ be the partition of $\IBr (G)$ by $\{ \{1_{G^\circ} \}, \IBr (G/M) - \{ 1_{G^\circ} \}$, $\IBr (G) - \IBr (G/M) \}$.  The Brauer character $\rho_{G/M}^o - \Phi_{1_{G^\circ}} (1) 1_{G^\circ}$, where $\rho_{G/M}^o$ is the restriction to $G^\circ$ of the inflation to $G$ of the regular character $\rho_{G/M}$ of $G/M$, has irreducible constituents in $\IBr (G/M) - \{ 1_{G^\circ} \}$, and it is not difficult to see that this Brauer character has the value $|G:M| - \Phi_{1_G^\circ} (1)$ on elements in $M - \{ 1 \}$ and the value $-\Phi_{1_G^\circ} (1)$ on elements in $G^\circ - M$.  On the other hand, the Brauer character $\rho_G^\circ-\rho_{G/M}^o$, where $\rho_G^\circ$ is the restriction of the regular character of $G$ to $G^\circ$, has irreducible constituents in $\IBr (G)- \IBr (G/M)$ and it has the value $0$ on elements of $G^\circ - M$ and the value of $-|G:M|$ on elements of $M - \{ 1 \}$.

We conclude that $(\mathfrak {R}, \mathfrak {S})$ is a super-Brauer character theory for $G$.  Since $G$ has exactly two super-Brauer character theories, $(\mathfrak{R},\mathfrak {S})$ must be either $\mathfrak {m} (G^\circ)$ or $\mathfrak {M} (G^\circ)$.  Because $G^\circ-M$ is not empty, we see that $(\mathfrak{R}, \mathfrak {S}) \neq \mathfrak {M} (G^\circ)$.  We deduce that $(\mathfrak{R}, \mathfrak {S}) = \mathfrak{m} (G^\circ)$, and this implies that $M - \{ 1 \}$ and $G^\circ - M$ are both conjugacy classes of $G$.  We conclude that $G$ has three $p$-regular conjugacy classes.

Thus, we may assume that $G^\circ = M$.  It follows that $M$ is a normal Hall $p$-complement of $G$.  In light of Theorem \ref{Hall}, we have that $M$ has exactly two $P$-invariant supercharacter theories.

Conversely, if $G$ has three $p$-regular classes, then it is obvious that $\mathfrak{m} (G^\circ)$ and $\mathfrak{M} (G^\circ)$ are the only possible super-Brauer character theories.  On the other hand, if $G$ has a normal $p$-complement $M$ and $M$ has exactly two $P$-invariant supercharacter theories,
then $G$ has two super-Brauer character theories by Theorem \ref{Hall}.
\end{proof}

Thus, the question of classifying the $p$-solvable groups with exactly two super-Brauer character theories is equivalent to determining the $p$-solvable groups with exactly three $p$-regular conjugacy classes and determining which $p$-solvable groups have a normal $p$-complement that has exactly two supercharacter theories that are invariant under the action of a Sylow $p$-subgroup.  Note that these two cases are not disjoint.

\section{Two $P$-invariant supercharacter theories }

As in the general question of determining the nonsolvable groups with two super-Brauer character theories, determining the nonsolvable $p'$-groups that are acted by a $p$-group having only two invariant supercharacter theories seems out of reach at this time.  Thus, we will focus on the solvable groups.  If $M$ is such a normal $p'$-subgroup, then we see from Theorem \ref{two theories} that $M$ will be minimal normal in $G$, and being solvable, this implies that $M$ is an elementary abelian $q$-group for some prime $q \neq p$.  We now show that either $p$ or $q$ must be $2$ when there are more than three orbits on the action of $P$.

\begin{lem} \label{cop lemma}
Let $p$ and $q$ be distinct primes, let $V$ be an elementary abelian $q$-group, and let $P$ be a $p$-group that acts via automorphisms irreducibly on $V$.  If $V$ has exactly two $P$-invariant supercharacter theories, then one of the following occurs:  {\rm (1)} $V$ is a union of three $P$-orbits, {\rm (2)} $q = 2$, or {\rm (3)} $p=2$.
\end{lem}

\begin{proof}
We write $\Irr (V) = \{ 1 \} \cup O_{1} \cup O_2 \cup \cdots \cup O_n$, where the $O_i$'s are $P$-orbits.  Notice that if $n = 1$, then $V$ will have only one $P$-invariant supercharacter theory, and if $n = 2$, then $V$ is the union of three $P$-orbits.  Thus, we may assume that $n \ge 3$.  By the discussion on page 2360 of \cite{Ref4}, we see that the partition of $\Irr (V)$ by the $O_i$'s yields a supercharacter theory $\mathfrak {m}_P (V)$ for $V$.  As we mentioned before, a supercharacter theory of $V$ is $P$-invariant if and only if it contains $\mathfrak {m}_P (V)$ where containment is as Definition 3.5 of \cite{hen}.

We now define $\overline {O_i} = \{ \overline \theta \mid \theta \in O_i \}$ where $\overline \theta (g)$ is the complex conjugate of $\theta (g)$ for all $g \in V$.  It is not difficult to see that $\overline {O_i}$ will be a $P$-orbit for each $i$.  Let $\mathfrak{S}$ be the supercharacter theory of $V$ whose partition of $\Irr (V)$ consists $\{ \chi, \overline {\chi} \}$ as $\chi$ runs over $\Irr (V)$.  From \cite{hen}, we know that $\mathfrak {m}_P (V) \vee \mathfrak{S}$ (as defined on page 4425 of \cite{hen}) is a supercharacter theory where $\Irr (V)$ is partitioned by the sets $\{ \chi, \overline {\chi} \mid \chi \in O_i \}$ for the appropriate values of $i$.  If there exists an integer $j$ with $1 \le j \le n$ such that $O_{j}\neq \overline{O_j}$, then $\mathfrak {m}_P (V) \vee \mathfrak{S} \neq \mathfrak {m}_P (V)$.  Since $n \ge 3$,
we have that $\mathfrak {m}_P (V) \vee \mathfrak{S} \ne \mathfrak {M} (V)$.  Finally, since $\mathfrak {m}_P (V) \vee \mathfrak{S}$ contains $\mathfrak {m}_P (V)$, it is $P$-invariant.  Thus, $V$ has more than two $P$-invariant supercharacter theories which contradicts our assumption.  Therefore, we must have $O_{i} = \overline{O_i}$ for all $i = 1, \dots, n$.

If $q = 2$, then $\overline {\mu} = \mu$ for every character $\mu \in \Irr (V)$ since $V$ is elementary abelian.  On the other hand, if $q \ne 2$, then $\overline {\mu} \ne \mu$ for every character $1_V \ne \mu \in \Irr (V)$.  This implies that $\mu$ and $\overline {\mu}$ must be in the same $P$-orbit when $\mu \ne 1_V$, and this can only occur if $2$ divides $|O_i|$ for all $i$.  Since the $|O_i|$ are $p$-powers, we conclude that $p = 2$.
\end{proof}

We now consider the case where a $p$-group $P$ acts faithfully and irreducibly on an elementary abelian $q$-group $V$ so that $V$ has one or two $P$-invariant supercharacter theories.

In Theorem B of \cite{Ninomiya199103} (5)-(10) give the examples where $P$ acts on $Q$ having exactly three orbits and are included in the following theorem as (1) (a)-(d), (2) when $r = 2$, and (3).  We note that we determined the sixteen $2$-groups that are referred to in (1) (c) and appear in Tables \ref{table 1a} and \ref{table 1b} using the computer algebra system Magma \cite{magma}.  We note that the last group in that table is a Sylow $2$-subgroup of $\GL_4 (3)$ and the other groups are subgroups of that Sylow $2$-subgroup that have the same orbits on $V$.  In addition, these groups correspond with the groups in Theorem B (9) of \cite{Ninomiya199103}.

\begin{thm}\label{three P-invariant}
Let $p$ and $q$ be distinct primes, let $V$ be an elementary abelian $q$-group, and let $P$ be a nontrivial $p$-group that acts faithfully and irreducibly on $V$.
Then $V$ has two $P$-invariant supercharacter theories if and only if  one of the following is true:
    \begin{enumerate}
    \item $p = 2$ and one of the following holds:
       \begin{enumerate}
       \item $q$ is a Fermat prime, $|V| = q$, and $P$ is the unique subgroup of index $2$ in ${\rm Aut} (V)$;
       \item $q = 3$, $|V| = 3^2$, and $P$ is cyclic of order $4$ or $P$ is dihedral of order $8$;
       \item $q = 3$, $|V| = 3^4$, and $P$ is a subgroup of $\GL_4 (3)$ that is conjugate one of the sixteen $2$-groups in Tables \ref{table 1a} and \ref{table 1b}.
       \item $q \ge 5$ is a Fermat prime, $|V| = q^2$, and $P$ is either a Sylow $2$-subgroup group $T$ of $\GL_2 (q)$ (i.e., $Z_{q-1} \wr Z_2$) or $P$ is the nonabelian, two generated subgroup of index $2$ in $T$.
       \item $q = 7$, $|V| = 7^2$, and $P$ is one of the following: cyclic of order $16$, generalized quaternion of order $16$, or semidihedral of order $32$ (i.e., a Sylow $2$-subgroup of ${\rm GL}_2 (7)$).
       \end{enumerate}
    \item $p$ and $q$ are primes so that there exists a positive integer $n$ and a prime $r$ so that $q - 1 = p^n r$, $|V| = q$, and $P$ is a Sylow $p$-subgroup of ${\rm Aut} (V)$.  (Note that one of $p$ or $r$ must be $2$.)
    \item $p \ne 2, 3$, $q = 3$, $l$ is an integer so that $3^l - 1 = 2 p^n$ for some integer $n$, $|V| = 3^l$ and $P \cong Z_{p^n}$.
    \end{enumerate}
\end{thm}

\begin{proof}
We now assume that $V$ has two $P$-invariant supercharacter theories.  Suppose first that $|V| = q$.  Since $P$ is nontrivial, this implies that $p$ divides $q - 1$, so $q \ge 3$, and thus, $q$ is odd.  By Lemma \ref{cop lemma}, this implies that either $P$ has three orbits on $V$ or $p=2$,  In Theorem 9.1 of \cite{hen}, Hendrickson describes all the supercharacter theories of a cyclic group.  Since $V$ has prime order, we see from that theorem that all of the supercharacter theories of $V$ have the form $\mathfrak {m}_A (V)$ where $A$ is a subgroup of ${\rm Aut} (V)$.  Hence, the $P$-invariant supercharacter theories will have the form $\mathfrak {m}_A (V)$ where $P \le A \le {\rm Aut} (V)$.  Therefore, $V$ has exactly two $P$-invariant supercharacter theories if and only if $P$ is a maximal subgroup of ${\rm Aut} (V)$.

Since $V$ is cyclic of odd prime order, ${\rm Aut} (V)$ is cyclic of order $q-1$, and so, $P$ is a maximal subgroup if and only if it has prime index in ${\rm Aut} (V)$.  If $q-1$ is a power of $2$, then $q$ is a Fermat prime, and ${\rm Aut} (V)$ is cyclic $2$-group, so $P$ is the unique subgroup of index $2$.  Otherwise, $q - 1$ is not a power of $2$, and the only way a $p$-subgroup can be maximal in ${\rm Aut} (V)$ is if it is the Sylow $p$-subgroup and there exists a positive integer $n$ and a prime $r$ so that $q-1 = p^n r$.  This completes the case where $|V| = q$.

Thus, we may assume that $|V| = q^n$ where $n>1$.  Suppose that $V$ is primitive as a module for $P$.  Observe that every abelian normal subgroup of $P$ has to act faithfully and homogeneously on $V$, and so every abelian normal subgroup of $P$ is cyclic.  By Corollary 1.10 (v) and (vi) of \cite{MaWo}, we know that $P$ has a cyclic subgroup of index dividing $2$.  Then applying Corollary 2.3 of \cite{MaWo}, we see that we may view $P$ as a subgroup of $\Gamma (V)$, the semi-linear group of $V$.  (The definition of $\Gamma$ can be found on pages 37-38 of \cite{MaWo}.)  Let $F = \Gamma_0 (V)$ and note that $F$ is cyclic of order $q^n - 1$ and $|\Gamma (V):F| = n$.  If $P$ is abelian, then Theorem 2.1 of \cite{MaWo} implies that $P \le F$.  If $P$ is not abelian, then arguing as in the proof of Lemma 2.3 of \cite{MaWo} and applying Corollary 1.10 (v), Lemma 2.2, and Theorem 2.1 of \cite{MaWo}, we have that $F \cap P$ is the cyclic subgroup of index $2$ in $P$.  Notice that since $P$ is a $p$-group, this implies that when $P$ is nonabelian, then $p = 2$, and $2$ divides $n$.

If $|F \cap P| = |V| - 1$, then $F \cap P$ acts transitively on $\Irr (V) - \{ 1_V \}$, and so, $V$ would have one $P$-invariant super-character theory.  Thus, we know that $|F \cap P| < |V| - 1$.  Since $|F \cap P| < |V| - 1$, we have that $F$ is not contained in $P$.  Recall that we may view $V$ as being the additive group of a field, $F$ as being a subgroup of the multiplicative group of the field, and if we take $C$ to be the Galois group for $F$ over $Z_q$, then $C$ is a complement for $F$ in $\Gamma (V)$.

It is known that orbits for $F \cap P$ on $V$ correspond to the cosets of $F \cap P$ as a subgroup of $F$ and that $C$ permutes the cosets of $F \cap P$ in $F$.  Also, an element $c$ of $C$ will fix a coset of $F \cap P$ if and only if the coset has a nonempty intersection with the subgroup of $F$ that corresponds to the fixed field of $c$.  Thus, $C$ fixes a coset of $F \cap P$ if and only if the coset has a nonempty intersection with the subgroup of $F$ corresponding to the fixed field of $C$.  Note that $Z_q$ is the fixed field for $C$, and so, the corresponding subgroup of $F$ will have order $q-1$.

We now suppose that $P$ is abelian, and hence cyclic.  If the subgroup of order $q-1$ in $F$ does not supplement $P$, then there will be cosets of $P$ that are not fixed by $C$.  Hence, $\mathfrak {m}_{PC} (V)$ will be different from both $\mathfrak {m}_P (V)$ and $\mathfrak {M} (V)$.  It follows that $V$ has at least three $P$-invariant supercharacter theories.  Thus, we may assume that $F$ is the product of $P$ and the cyclic group of order $q-1$.  Since $P < F$, this implies that $q \ne 2$, and so by Lemma \ref{cop lemma}, we know that $p = 2$. Also, note that if $P$ is not a maximal subgroup of $F$, then there exists a subgroup $P < E < F$, and again it is not hard to see that $\mathfrak {m}_{E} (V)$ will be different from both $\mathfrak {m}_P (V)$ and $\mathfrak {M} (V)$ and that $V$ has at least three $P$-invariant supercharacter theories.  Thus, $P$ will be a maximal subgroup of $F$, and thus, $|F:P|$ is a prime.  Hence, $q-1$ is the product of a prime with a power of $2$, and $(q^n-1)/(q-1)$ is a power of $2$.   By the Zsigmondy prime theorem, we see that $n = 2$ and $q$ is a Mersenne prime.  Note that $q+1$ is now a power of $2$.  If $q > 3$, then since $3$ does not divide $q$ or $q + 1$, we must have that $3$ divides $q-1$.  Also, $4$ divides $q+1$, so $4$ cannot divide $q - 1$.  Hence, the only possibilities for $q$ are $q = 3$ and $q = 7$.  When $q=3$, we deduce that $P$ is cyclic of order $4$ and when $q = 7$, we see that $P$ is cyclic of order $16$.  This completes the case when $P$ is cyclic.

Thus, we have the case when $V$ is primitive for $P$ and $P$ is not abelian.  This implies that $p = 2$ and $2$ divides $n$.  Also, $F \cap P$ has index $2$ in $P$.  Since $F \cap P$ is cyclic, this implies that $P$ is either dihedral, semi-dihedral, or generalized quaternion.  Notice that every $P$ orbit on $V$ will be either an $F \cap P$-orbit or the union of two $F \cap P$ orbits.  Also, $P$ is a Frobenius complement if and only if $P$ is generalized quaternion group, and in that case, every $P$ orbit will be the union of two $F \cap P$-orbits.  In the other cases, $P$ will have at least one orbit that is an $F \cap P$-orbit.

Suppose that $F \cap P$ is not maximal in $F$, so there exists a group $N$ so that $F \cap P < N < F$ and $|N:F \cap P|$ is prime.  Notice that each $N$-orbit will be a union of $|N:F \cap P|$ orbits of $F \cap P$, and the $PN$-orbits will be unions of $N$-orbits.  Notice that $\mathfrak {m}_{NP} (V)$ will equal $\mathfrak {M} (V)$ if and only if $PN$ acts transitively on $V - \{ 0 \}$.  This will imply that $q^n - 1$ is a prime times a power of $2$.  By the Zsigmondy prime theorem, this implies that $q$ is either a Mersenne prime or a Fermat prime and $n = 2$.  Also, we see that one of $q+1$ and $q-1$ is a power of $2$ that is at least $4$.  It follows that the other one cannot be divisible by $4$.  On the other hand, $3$ must divide one of $q-1$, $q$, or $q+1$; so if $q > 1$, then one of $q-1$ and $q+1$ which is not a power of $2$ is divisible by $6$ and not $12$.  It follows that the only possibilities for $q$ are $3$, $5$, and $7$.  Also, we see that $PN$ acts Frobeniusly on $V$, so $P$ is generalized quaternion, so $q^2 - 1 = |PN| = |F|$.  In particular, if $q = 3$, then $|PN| = 8$ and $|P| = 4$ which contradicts the fact that $P$ is nonabelian.  If $q = 7$, then $|PN| = 48$ and $|P| = 16$, so $P$ is the generalized quaternion group of order $16$.  If $q = 5$, then $|PN| = 24$ and $|P| = 8$, so $P$ is the quaternions, and the $V - \{ 0 \}$ consists of three $P$-orbits of size $2$.  However, $P$ is properly contained in a Sylow $2$-subgroup $T$ of ${\rm GL}_2 (5)$, and we will see later that $T$ has one orbit of size $8$ and one orbit of size $16$ on $V - \{ 0 \}$.  Thus, $\mathfrak {m}_T (V)$ is neither $\mathfrak {m}_P (V)$ nor $\mathfrak {M} (V)$, so $V$ has at least three $P$-invariant supercharacter theories in this case.

We now assume that $PN$ does not act transitively on $V - \{ 0 \}$, so $\mathfrak {m}_{NP} (V)$ is different than $\mathfrak {M} (V)$.  If $|N:F \cap P| > 2$, it follows that the $PN$-orbits will be different than the $P$-orbits, so $\mathfrak {m}_{PN} (V)$ will be different than $\mathfrak {m}_P (V)$, and we see that there exist at least three $P$-invariant supercharacter theories for $V$, a contradiction.  Now suppose that $|N: F \cap P| = 2$, so that $PN$ will be a $2$-group.  Note that every $N$-orbit is the union of two $F \cap P$-orbits, and every $PN$-orbit will be either an $N$-orbit or the union of two $N$-orbits.  The only way every $PN$-orbit could be a $P$-orbit is if every $P$-orbit is the union of two $P \cap N$-orbits, so $P$ must be a generalized quaternion group, and every $PN$-orbit must be an $N$-orbit, so $PN$ is not a Frobenius complement.  Notice that a dihedral group cannot have a subgroup that is a generalized quaternion group, so $PN$ must be a semi-dihedral group.  Also, all of the $PN$-orbits on $V$ have the same size, so $PN$ acts half-transitively on $V$.  By Theorem II of \cite{half}, we have that $q$ is a Mersenne prime, and $|PN|$ has the same size as a Sylow $2$-subgroup of $\Gamma (V)$, so $PN$ is a Sylow $2$-subgroup of $\Gamma (V)$.  If $|F:N|$ is not prime, then we can find $N_1$ so that $N < N_1 <F$.  Using the arguments above, we see that $\mathfrak {m}_{N_1P} (V)$ is neither $\mathfrak {M} (V)$ nor $\mathfrak {m}_P (V)$.  Now, $|F:N|$ is a prime implies that $q = 7$ and hence, $P$ is generalized quaternion of order $16$.

We now have the case that $|F:N \cap P|$ is a prime.  As above, this implies that $|V| = q^2$ and $q$ is either $3$, $5$, or $7$.  If $q = 3$, then $P$ has index $2$ in $\Gamma (V)$ and $|P| = 8$.  Since $P$ is nonabelian, this implies that $P$ is either dihedral or quaternion.  However, if $P$ is quaternion, then $P$ acts transitively on $V$, which is a contradiction.  Thus, $P$ is dihedral of order $8$.  If $q$ is $5$ or $7$, this implies that $P$ is a Sylow $2$-subgroup of $\Gamma (V)$.  If $q = 7$, then using order considerations, we conclude that $P$ is a Sylow $2$-subgroup of $\GL_2 (7)$.  Note that when $q  = 5$, this implies that $P$ is semi-dihedral of order $16$.  This completes the case where $V$ is primitive as a module for $P$.

We now suppose that $V$ is not primitive as a module for $P$.  In particular, we can write $V = W_1 \oplus \dots \oplus W_p$ so that $P$ transitively permutes the $W_i$'s (see Corollary 0.3 of \cite{MaWo}).  We now assign weights to the elements of $V$ and the characters in $\Irr (V)$.  The weight of an element $v = (w_1, \dots, w_p)$ with $w_i \in W_i$ is the number of nonzero $w_i$'s.  Similarly, the weight of $\chi = (\lambda_1, \dots, \lambda_p)$ for $\lambda_i \in \Irr (W_i)$ is the number of nonprincipal $\lambda_i$'s.  For $i = 0, \dots, p$, let $B_i$ be the set of elements of weight $i$ in $V$ and let $C_i$ be the set of characters of weight $i$ in $\Irr (V)$.  Note that $B_0 = \{ 0 \}$ and $C_0 = \{ 1_V \}$.  Note that $V$ is partitioned by $\mathcal {B} = \{ B_0, B_1, \dots, B_p \}$ and $\Irr (V)$ is partitioned by $\mathcal {C} = \{ C_0, C_1, \dots, C_ p \}$.  We will show that $(\mathcal {C}, \mathcal {B} )$ is a super-character theory for $V$.  Note that the action of $P$ on $V$ preserves the weight of both the elements of $V$ and the characters in $\Irr (V)$, so at $(\mathcal {C}, \mathcal {B} )$ is $P$-invariant.

We have $\{ 0 \} = B_0 \in \mathcal {B}$ and $|\mathcal {B}| = p + 1 = |\mathcal {C}|$.  Thus, it suffices to find a character for each $i$ with irreducible constituents in $C_i$ that is constant on all the $B_j$'s.  For each $i$, take $\rho^*_i$ to be the regular character of $W_i$ minus the principal character of $W_i$.  Thus, $\rho^*_i (0_i) = |W_i|-1$ and $\rho^*_i (w) = -1$ for $w \in W_i \setminus \{ 0_i \}$.  Given a subset $S$ of $\{ 1, \dots, p \}$, we define $\rho^*_S$ to be the character whose $i$th component is $\rho^*_i$ when $i \in S$ and $1_{W_i}$ when $i$ is not in $S$.  Notice that the irreducible constituents of $\rho^*_S$ will have weight $|S|$ and hence will lie in $C_{|S|}$.  We define $\sigma_0$ to be $1_V$, and for $1 \le i \le p$, we define $\sigma_i$ to be the sum of $\rho^*_S$ over all subsets $S$ of size $i$ in $\{ 1, \dots, p \}$.

Given an element $v = (w_1, \dots, w_p) \in V$, we define the support, ${\rm supp} (v)$, to be the set of $i$'s where $w_i \ne 0_i$.  In computing $\rho^*_S (v)$ notice that we obtain a $-1$ for each element of $S \cap {\rm supp} (v)$, a $q^m-1$ for each element in $S \setminus {\rm supp} (v)$, and a $1$ for the remaining components.  It follows that $\rho^*_S (v) = (-1)^{|S \cap {\rm supp} (v)|} (q^m - 1)^{|S| - |S \cap {\rm supp} (v)|}$ where $|W_i| = q^m$.  Thus, $\sigma_i (v)$ will be determined by the various sizes of the intersections of ${\rm supp} (v)$ with subsets of $\{ 1, \dots, p \}$ of size $i$.  Since this is completely determined by the size of ${\rm supp} (v)$, it follows that $\sigma_i$ is constant on the $B_j$'s.  In particular, we conclude that $(\mathcal {C}, \mathcal {B})$ is a super-character theory for $V$.  Since $|\mathcal {B}| \ge 3$, we know that $(\mathcal {C}, \mathcal {B})$ is not $\mathfrak {M} (V)$.

For $(\mathcal {C}, \mathcal {B})$ to be $\mathfrak{m}_P (V)$, it must be that $P$ acts transitively on each $B_i$.  Notice that the size of $B_2$ will be $(q^m-1)^2 p (p-1)/2$, and this can be a power of $p$ only when $p$ is $2$ or $3$.  Let $N \le P$ be the stabilizer in $P$ of $W_1$, so $|P:N| = p$.  This implies that $N$ is a normal subgroup of $P$, and since the stabilizers of the $W_i$'s are all conjugate in $P$, it follows that $N$ stabilizes all of the $W_i$'s.  Note that $C_P (w) \subseteq N$ for all elements of $w \in W$ having weight $1$.  Let $H = N/C_N(W_1)$, and by Lemma 2.8 of \cite{MaWo}, we know that $P$ is a subgroup of $H \wr Z_p$.  We also know that $H$ acts transitively on $W_1 \setminus \{ 0_1 \}$.  By Theorem 6.8 of \cite{MaWo}, we see that $H \le \Gamma (W_1)$.

If $p = 3$, then $q = 2$, and $m = 2$.  We see that $|V| = 2^6$.  We have that $|B_1| = 3 \cdot 3 = 9$ and $|B_2| = 3^2 \cdot 3 = 27$ and $|B_3| = 3^3 = 27$.  Notice that $|W_1| = 4$ and so, $\Gamma (W_1) \cong S_3$.  Since $H$ is a $3$-subgroup, we see that $H \cong Z_3$ and so $P$ is a subgroup of $Z_3 \wr Z_3$.  Notice by order considerations, we see that $Z_3 \wr Z_3$ is isomorphic to a Sylow $3$-subgroup of $\GL_6 (2)$.  Also, we see that $P$ must have at least order $27$, so either $P$ is $Z_3 \wr Z_3$ or $P$ is a maximal subgroup of $Z_3 \wr Z_3$.  Using Magma, we find the orbits of a Sylow $3$-subgroup $T$ of $\GL_6 (2)$ and its maximal subgroups on $V$, and from Magma, we see that the maximal subgroup having exponent $3$ has different orbits than $T$ and thus $V$ will have more than two $P$-invariant supercharacter theories if $P$ is one of these subgroups, but the two maximal subgroups that have exponent $9$ and are nonabelian have the same orbits as $T$.  Since they have order $27$, they must be extra-special.

We let $O_1$, $O_2$, and $O_3$ be the nonprincipal $P$-orbits of $V$, and let $\Theta_i$ be the supercharacter for $O_i$.  Let $V_1$, $V_2$, and $V_3$ be the nonzero $P$-orbits of $V$ and let $v_i$ be a representative of $V_i$.  Using Magma, we compute the values of the supercharacters for the orbits in $\Irr (V)$ on the orbits of $V$.  The values can be found in Table \ref{table 2^6}.

\begin{table}
$$\begin{tabular}{c|cccc}
 & 0 & $v_1$ & $v_2$ & $v_3$ \\
\hline
$1_V$ & 1 & 1 & 1 & 1 \\
$\Theta_1$ & 9 & 5 & 1 & -3 \\
$\Theta_2$ & 27 & 3 & -5 & 3 \\
$\Theta_3$ & 27 & -9 & 3 & - 1
\end{tabular}$$
\caption{$P$-invariant supercharacters when $|V| = 2^6$}\label{table 2^6}
\end{table}

We see that the only supercharacter for an orbit in $\Irr (V)$ that has the same value on two orbits for $V$ is $\Theta_2$ and the two orbits in $V$ are $V_1$ and $V_3$.  Observe that $\Theta_1 + \Theta_3$ has the value $-4$ on both $v_1$ and $v_3$.  Thus, we obtain another supercharacter theory by merging $O_1$ with $O_3$ and $V_1$ with $V_3$.  This gives three $P$-invariant supercharacter theories.

Now suppose that $p = 2$.  We see that $|W_1| - 1 = q^m - 1$ is a power of $2$.  This implies that either $q$ is a Fermat prime and $m = 1$ or $q = 3$ and $m = 2$.  Suppose first that $|W_1| = q$.  This implies that $H$ is cyclic of order $q - 1$ and that $P$ is a subgroup of $Z_{q-1} \wr Z_2$.  Notice that $|B_1| = 2 (q-1)$ and $|B_2| = (q-1)^2$. In particular, $(q-1)^2$ divides $|P|$.  The size of a Sylow $2$-subgroup of ${\rm GL}_2 (q)$ is $2(q-1)^2$.  Also, one can see that a Sylow $2$-subgroup of ${\rm GL}_2 (q)$ must contain a copy of $Z_{q-1} \wr Z_2$.  Therefore, a Sylow $2$-subgroup $T$ of ${\rm GL}_2 (q)$ is isomorphic to $Z_{q-1} \wr Z_2$, and hence, $P$ is either isomorphic to $T$ or is a maximal subgroup of $T$.  If $q = 3$, then $T = {\rm GL}_2 (3)$, and it acts transitively on $V - \{ 0 \}$.  Thus, $P$ must be a maximal subgroup.  Note that $P$ has three maximal subgroups, one of which is cyclic and acts transitively and one of which is the quaternions which also acts transitively.  The remaining maximal subgroup is the dihedral group of order $8$ which must be $P$.  Suppose now that $q > 3$.  It is not difficult to see that $B_0$, $B_1$, and $B_2$ are the three orbits for $T$ acting on $V$.  It follows that $P$ must have these three orbits, and the result now follows from Theorem B of \cite{Ninomiya199103}.  The remaining case is that $|W_1| = 9$.  In this case, if $T$ is a Sylow $2$-subgroup of ${\rm GL}_4 (3)$, then $T \cong S_{16} \wr Z_2$ where $S_{16}$ is the semidihedral group of order $16$.  Again, it is not difficult to see that $B_0$, $B_1$, and $B_2$ are the three orbits of $T$ on $V$.  Thus, $P$ will need to be a subgroup of $T$ that has the same orbits as $T$.  Again, one can appeal to Theorem B of \cite{Ninomiya199103} or one can compute these subgroups in Magma.  This completes the proof when $V$ has two $P$-invariant character theories.

We now suppose that if $P$ and $V$ are any of the given groups, then $V$ has two $P$-invariant supercharacter theories.  Observe that if $P$ has three orbits on $\Irr (V)$, then the only possible $P$-invariant supercharacter theories are $\mathfrak{m}_P (V)$ and $\mathfrak {M} (V)$.  Note that this handles (1) (a)-(d), (2) when $r = 2$, and (3).  Also, in (2) when $|V| = q$, $p=2$, and $r$ is odd; we can use Theorem 9.1 of \cite{hen} to see that the $p$-invariant supercharacters of $V$ are in bijection with the subgroups of ${\rm Aut} (V)$.  Since $q-1 =
2^n r$, it follows that $P$ is a maximal subgroup of ${\rm Aut} (V)$, and this yields the conclusion.

In the remaining case, (1) (e), we see that $P$ has four orbits on $\Irr (V)$.  One orbit is $\{ 1_V \}$ and three other orbits.  If there is a $P$-invariant supercharacter theory $(\mathcal {X}, \mathcal {Y})$ other than $\mathfrak {m}_P (V)$ and $\mathfrak {M} (V)$, then $\mathcal {X}$, the partition of $\Irr (V)$, must have a union of two of these nonprincipal orbits as one set $X_1$ and the other nonprincipal set $X_2$ will be the remaining nonprincipal orbit.  Similarly, the partition $\mathcal {Y}$ of $V$ will have one set $Y_1$ that is the union of two of the nonzero $P$-orbits of $V$ and the other set $Y_2$ will be a nonzero $P$-orbit.  Thus, the supercharacters $\chi_1$ for $X_1$ and $\chi_2$ for $X_2$ will have to have the same value on both of the $P$-orbits that make up $Y_1$.

As above, we let $O_1$, $O_2$, and $O_3$ be the nonprincipal $P$-orbits of $V$, and let $\Theta_i$ be the supercharacter for $O_i$.  Let $V_1$, $V_2$, and $V_3$ be the nonzero $P$-orbits of $V$ and let $v_i$ be a representative of $V_i$.  Using Magma, we compute the values of the supercharacter for the orbits in $\Irr (V)$ on the orbits for $V$.  The values appear in Table \ref{table (1)(e)}

\begin{table}
$$\begin{tabular}{c|cccc}
 & 0 & $v_1$ & $v_2$ & $v_3$ \\
\hline
$1_V$ & 1 & 1 & 1 & 1 \\
$\Theta_1$ & 16 & $A_1$ & $A_2$ & $A_3$ \\
$\Theta_2$ & 16 & $A_2$ & $A_3$ & $A_1$ \\
$\Theta_3$ & 16 & $A_3$ & $A_1$ & $A_2$
\end{tabular}$$
\caption{$P$-invariant Supercharacters in (1)(e)}\label{table (1)(e)}
\end{table}

Where $A_1 = 3 \zeta^5 + \zeta^4 + \zeta^3 + 3\zeta^2 + 1$, $A_2 = -\zeta^5 + 2 \zeta^4 +2 \zeta^3 - \zeta^2$, and $A_3 = -2 \zeta^5 - 3\zeta^4 - 3\zeta^3 - 2 \zeta^2 -2$ and $\zeta$ is a primitive $7$ root of unity.  Notice that none of the supercharacters for the orbits in $\Irr (V)$ have the same value on two orbits for $V$, so it is not possible to have another $P$-invariant supercharacter theory.
\end{proof}

In the introduction, we mentioned that we would display groups with three super-Brauer character theories that have more than three $p$-regular conjugacy classes.  We now use Theorem \ref{three P-invariant} to produce these examples. The first example is $G = V \rtimes P$ where $V$ is elementary abelian and $P$ is either cyclic of order $16$ or semidihedral of order $32$.  The second example also has $G = V \rtimes P$, but in this case $|V| = q$ where $q$ is a prime with the property that $q - 1 = 2^l r$ for some integer $l$ and odd prime $r$.  In this case, $P$ is a cyclic group of order $2^l$.  Examples of primes $q$ with this property are $q = 7, 11, 13, 23, 29, 39, 41, \dots$.

Finally, we note that we can apply Theorem \ref{two theories} and Theorem \ref{three P-invariant} to see that any solvable group with three super-Brauer character theories must either have three $p$-regular classes or must be one of the groups mentioned in the previous paragraph.  For completeness, we have listed $G/{\bf O}_p (G)$ for all of the groups that have two super-Brauer character theories.  We have listed these groups in three tables.  The first two tables have the groups of the form $V \rtimes P$ where $V$ is an elementary abelian $q$-group for some prime $q$ and $P$ is a $p$-group that acts faithfully and irreducibly on $V$.  In Table \ref{table 2a}, we have the ones where the action of $P$ on $V$ has three orbits.  Table \ref{table 2b} includes the ones where the action of $P$ has more than three orbits on $V$.  Notice that these are precisely the groups with two super-Brauer character theories that have more than three $p$-regular conjugacy classes.  Finally, in Table \ref{table 2c}, we list the groups with three $p$-regular classes that do not have the form $V \rtimes P$ for an elementary abelian $q$-group $V$ and a nontrivial $p$-group $P$.  The groups with three $p$-regular conjugacy classes were classified in Theorem B of \cite{Ninomiya199103} and the Theorem in \cite{Ninomiya199301}.  As we mentioned above, we used the computer algebra system Magma, \cite{magma}, to find the $2$-groups that act on the elementary abelian group of order $81$ with three orbits.  We also use Huppert's classification of groups that act transitively on the nonidentity elements of an elementary abelian group found as Theorem 6.8 in \cite{MaWo} in place of Ninomiya's description of these groups.  Note that it appears that Ninomiya missed one group in his classification.  In particular, he missed the group $(\mathbb{Z}_3)^4 \rtimes S$ where $F(S)$ is an extra-special group of order $2^5$, the quotient $F_2(S)/F(S)$ has order $5$, and the quotient $S/F_2(S)$ is cyclic of order $4$ that is described in conclusion (a) of Theorem 6.8 of \cite{MaWo}.  We have double-checked using Magma the groups of the form $V \rtimes S$ where $V$ is elementary abelian of order $25$ or $81$ and $S$ acts transitively on $V - \{ 0 \}$.

% ------------------------------------------------------------------------

\section*{Acknowledgments}
The first author thanks the support of China Scholarship Council,
Department of Mathematical Sciences of Kent State University for its hospitality,
Funds of Henan University of Technology (2014JCYJ14, 2016JJSB074),
the Project of Department of Education of Henan Province (17A110004),
the Projects of Zhengzhou Municipal Bureau of Science and Technology (20150249, 20140970),
and the NSFC (11571129).
% ------------------------------------------------------------------------

% ------------------------------------------------------------------------

\newpage

%\section*{Appendix}
\begin{table}
$$\begin{tabular} {c | ccc}
$|P|$ & A & B & C\\
\hline
\medskip
$2^6$ & $\left[ \begin{array}{cccc} 2 & 2 & 0 & 0 \\ 2 & 1 & 2 & 2 \\ 2 & 1 & 2 & 1 \\ 0 & 0 & 2 & 1 \end{array} \right]$ & $\left[ \begin{array}{cccc} 0 & 0 & 2 & 2 \\ 1 & 1 & 1 & 1 \\ 1 & 0 & 1 & 0 \\ 0 & 0 & 2 & 0 \end{array} \right]$ & \\
\smallskip
$2^6$ & $\left[ \begin{array}{cccc} 2 & 2 & 0 & 0 \\ 2 & 1 & 1 & 1 \\ 2 & 1 & 2 & 1 \\ 0 & 0 & 2 & 1 \end{array} \right]$ & $\left[ \begin{array}{cccc} 0 & 0 & 2 & 2 \\ 1 & 1 & 1 & 1 \\ 1 & 0 & 1 & 0 \\ 0 & 0 & 2 & 0 \end{array} \right]$ & \\
\smallskip
$2^6$ & $\left[ \begin{array}{cccc} 1 & 2 & 2 & 1 \\ 0 & 1 & 1 & 2 \\ 1 & 1 & 1 & 0 \\ 0 & 0 & 2 & 0 \end{array} \right]$ & $\left[ \begin{array}{cccc} 1 & 0 & 1 & 1 \\ 1 & 2 & 2 & 2 \\ 1 & 0 & 0 & 1 \\ 0 & 0 & 2 & 1 \end{array} \right]$ & \\
\smallskip
$2^6$ & $\left[ \begin{array}{cccc} 1 & 0 & 1 & 1 \\ 1 & 1 & 0 & 1 \\ 1 & 0 & 0 & 1 \\ 0 & 0 & 2 & 1 \end{array}
\right]$ & $\left[ \begin{array}{cccc} 1 & 2 & 2 & 1 \\ 1 & 1 & 1 & 2 \\ 2 & 1 & 0 & 2 \\ 2 & 0 & 0 & 1 \end{array} \right]$ & \\
\smallskip
$2^7$ & $\left[ \begin{array}{cccc} 2 & 2 & 0 & 0 \\ 2 & 1 & 1 & 1 \\ 2 & 1 & 2 & 1 \\ 0 & 0 & 2 & 1 \end{array}
\right]$ & $\left[ \begin{array}{cccc} 0 & 0 & 1 & 1 \\ 1 & 2 & 1 & 0 \\ 1 & 0 & 0 & 2 \\ 0 & 0 & 0 & 1 \end{array}
\right]$ & \\
\smallskip
$2^7$ & $\left[ \begin{array}{cccc} 2 & 2 & 0 & 0 \\ 2 & 1 & 2 & 2 \\ 2 & 1 & 2 & 1 \\ 0 & 0 & 2 & 1 \end{array}
\right]$ & $\left[ \begin{array}{cccc} 0 & 0 & 1 & 1 \\ 1 & 2 & 1 & 0 \\ 1 & 0 & 0 & 2 \\ 0 & 0 & 0 & 1 \end{array}
\right]$ & \\
\smallskip
$2^7$ & $\left[ \begin{array}{cccc} 1 & 0 & 1 & 1 \\ 1 & 1 & 0 & 1 \\ 1 & 0 & 0 & 1 \\ 0 & 0 & 2 & 1 \end{array}
\right]$ & $\left[ \begin{array}{cccc} 1 & 1 & 0 & 0 \\ 0 & 2 & 0 & 0 \\ 1 & 2 & 2 & 0 \\ 0 & 0 & 0 & 1 \end{array} \right]$ & $\left[ \begin{array}{cccc} 2 & 0 & 1 & 1 \\ 1 & 2 & 0 & 1 \\ 2 & 2 & 2 & 2 \\ 0 & 1 & 0 & 0 \end{array}
\right]$ \\
\smallskip
$2^7$ & $\left[ \begin{array}{cccc} 1 & 1 & 0 & 0 \\ 0 & 2 & 0 & 0 \\ 1 & 2 & 2 & 0 \\ 0 & 0 & 0 & 1 \end{array} \right]$ & $\left[ \begin{array}{cccc} 1 & 0 & 1 & 1 \\ 1 & 2 & 2 & 2 \\ 1 & 0 & 0 & 1 \\ 0 & 0 & 2 & 1 \end{array} \right]$ & $\left[ \begin{array}{cccc} 0 & 0 & 2 & 2 \\ 1 & 1 & 1 & 1 \\ 1 & 0 & 1 & 0 \\ 0 & 0 & 2 & 0 \end{array}
\right]$ \\
\smallskip
$2^7$ & $\left[ \begin{array}{cccc} 1 & 1 & 0 & 0 \\ 0 & 2 & 0 & 0 \\ 1 & 2 & 2 & 0 \\ 0 & 0 & 0 & 1 \end{array} \right]$ & $\left[ \begin{array}{cccc} 1 & 0 & 1 & 1 \\ 1 & 1 & 0 & 1 \\ 0 & 0 & 1 & 2 \\ 2 & 0 & 0 & 2 \end{array}
\right]$ & \\
\smallskip
$2^7$ & $\left[ \begin{array}{cccc} 1 & 0 & 1 & 1 \\ 1 & 2 & 2 & 2 \\ 0 & 0 & 1 & 2 \\ 2 & 0 & 0 & 2 \end{array}
\right]$ & $\left[ \begin{array}{cccc} 1 & 1 & 0 & 0 \\ 0 & 2 & 0 & 0 \\ 1 & 2 & 2 & 0 \\ 0 & 0 & 0 & 1 \end{array} \right]$ & \\

\end{tabular}$$
\caption{Subgroups $P = \langle A, B, C \rangle$ of $\GL_4 (3)$ with three orbits}\label{table 1a} \end{table}

\begin{table}
$$ \begin{tabular} {c | ccc}
$|P|$ & A & B & C\\
\hline
\medskip
$2^8$ & $\left[ \begin{array}{cccc} 1 & 1 & 0 & 0 \\ 0 & 2 & 0 & 0 \\ 1 & 2 & 2 & 0 \\ 0 & 0 & 0 & 1 \end{array} \right]$ & $\left[ \begin{array}{cccc} 1 & 0 & 1 & 1 \\ 1 & 2 & 2 & 2 \\ 1 & 0 & 0 & 1 \\ 0 & 0 & 2 & 1 \end{array}
\right]$ & $\left[ \begin{array}{cccc} 0 & 0 & 1 & 1 \\ 1 & 2 & 1 & 0 \\ 1 & 0 & 0 & 2 \\ 0 & 0 & 0 & 1 \end{array}
\right]$ \\
\smallskip
$2^8$ & $\left[ \begin{array}{cccc} 1 & 0 & 0 & 0 \\ 0 & 2 & 2 & 1 \\ 0 & 0 & 1 & 0 \\ 0 & 0 & 0 & 1 \end{array}
\right]$ & $\left[ \begin{array}{cccc} 2 & 2 & 0 & 0 \\ 2 & 1 & 1 & 1 \\ 2 & 1 & 2 & 1 \\ 0 & 0 & 2 & 1 \end{array}
\right]$ & \\
\smallskip
$2^8$ & $\left[ \begin{array}{cccc} 1 & 0 & 1 & 1 \\ 1 & 1 & 0 & 1 \\ 1 & 0 & 0 & 1 \\ 0 & 0 & 2 & 1 \end{array}
\right]$ & $\left[ \begin{array}{cccc} 1 & 1 & 0 & 0 \\ 0 & 2 & 0 & 0 \\ 1 & 2 & 2 & 0 \\ 0 & 0 & 0 & 1 \end{array}
\right]$ & $\left[ \begin{array}{cccc} 0 & 0 & 1 & 1 \\ 1 & 2 & 1 & 0 \\ 1 & 0 & 0 & 2 \\ 0 & 0 & 0 & 1 \end{array}
\right]$ \\
\smallskip
$2^8$ & $\left[ \begin{array}{cccc} 1 & 1 & 0 & 0 \\ 2 & 2 & 1 & 1 \\ 1 & 2 & 2 & 0 \\ 0 & 0 & 0 & 1 \end{array}
\right]$ & $\left[ \begin{array}{cccc} 1 & 0 & 1 & 1 \\ 1 & 2 & 2 & 2 \\ 1 & 0 & 0 & 1 \\ 0 & 0 & 2 & 1 \end{array}
\right]$ & \\
\smallskip
$2^8$ & $\left[ \begin{array}{cccc} 1 & 0 & 1 & 1 \\ 1 & 1 & 0 & 1 \\ 1 & 0 & 0 & 1 \\ 0 & 0 & 2 & 1 \end{array}
\right]$ & $\left[ \begin{array}{cccc} 2 & 2 & 0 & 0 \\ 2 & 1 & 1 & 1 \\ 2 & 1 & 2 & 1 \\ 0 & 0 & 2 & 1 \end{array}
\right]$ & \\
\smallskip
$2^9$ & $\left[ \begin{array}{cccc} 1 & 0 & 0 & 0 \\ 0 & 2 & 2 & 1 \\ 0 & 0 & 1 & 0 \\ 0 & 0 & 0 & 1 \end{array}
\right]$ & $\left[ \begin{array}{cccc} 1 & 1 & 0 & 0 \\ 0 & 2 & 0 & 0 \\ 1 & 2 & 2 & 0 \\ 0 & 0 & 0 & 1 \end{array}
\right]$ & $\left[ \begin{array}{cccc} 1 & 0 & 1 & 1 \\ 1 & 2 & 2 & 2 \\ 1 & 0 & 0 & 1 \\ 0 & 0 & 2 & 1 \end{array}
\right]$ \\
\end{tabular}$$
\caption{Subgroups $P = \langle A, B, C \rangle$ of $\GL_4 (3)$ with three orbits on $(\mathbb{Z}_3)^4$}\label{table 1b}
\end{table}

\begin{table}
$$\begin{tabular}{||c|c|c||}
\hline
\hline
$p$ & $G/{\bf O}_p (G)$ &  \\
\hline
\hline
$2$ & $(\mathbb{Z}_3 \times \mathbb{Z}_3)\rtimes P$ & $P$ is either $\mathbb{Z}_4$ or $D_8$  \\
\hline
$2$ & $\mathbb{Z}_q \rtimes \mathbb{Z}_{2^n}$ & $q = 2^{n+1} + 1$ is a Fermat prime for some positive integer $n$ \\
\hline
not $2$ & $\mathbb{Z}_q \rtimes \mathbb{Z}_{p^n}$ & $q = 2 p^n + 1$ is prime for some positive integer $n$ \\
\hline
not $2$ or $3$ & $(\mathbb{Z}_3)^l \rtimes \mathbb{Z}_{p^n}$ & $3^l = 2p^n + 1$ for some positive integer $l$ \\
\hline
$2$ & $(\mathbb{Z}_3)^4 \rtimes P$ & $P$ is conjugate to one of the groups in Tables \ref{table 1a} and \ref{table 1b}\\
\hline
$2$ & $(\mathbb{Z}_q \times \mathbb{Z}_q) \rtimes P$ & $q = 2^e + 1$ is a Fermat prime for some positive integer $e$ \\
& & $P$ is either a Sylow $2$-subgroup $T$ of $\GL_2 (q)$ or is conjugate \\
& & to the nonabelian, two generated subgroup of index $2$ in $T$ \\
\hline
\hline
\end{tabular}$$
\caption {Groups $V\rtimes P$ with three $p$-regular classes}\label{table 2a}
$$\begin{tabular}{|c|c|c|}
\hline
\hline
$p$ & $G/{\bf O}_p (G)$ &  \\
\hline
\hline
$2$ & $\mathbb{Z}_q \rtimes \mathbb{Z}_{2^n}$ & $q = r 2^n + 1$ is prime for some prime $r$ and positive integer $n$ \\
\hline
$2$ & $(\mathbb{Z}_7 \times \mathbb{Z}_7) \rtimes P$ & $P$ is either cyclic of order $16$, generalized quaternion \\
& & of order $16$, or semidihedral of order $32$ \\
\hline
\hline
\end{tabular}$$
\caption{Groups $V \rtimes P$ with more than three $p$-regular classes}\label{table 2b}
$$\begin{tabular}{|c|c|c|}
\hline
\hline
$p$ & $G/{\bf O}_p (G)$ &  \\
\hline
\hline
not $3$ & $\mathbb{Z}_3$ &  \\
\hline
not $2$ or $3$ & ${\rm Sym} (3)$ &  \\
\hline
$3$ & $\SL_2 (3)$ &  \\
\hline
$2$ & $E \rtimes P$ & $E$ is extraspecial of order $27$ and exponent $3$ \\
& & $P$ is either $\mathbb{Z}_8$ or $S_{16}$ \\
\hline
not $2$ & $\mathbb{Z}_q \rtimes (\mathbb{Z}_2 \times \mathbb{Z}_{p^n})$ & $q = 2 p^n + 1$ is prime \\
\hline
not $2$ or $3$ & $(\mathbb{Z}_3)^l \rtimes (\mathbb{Z}_2 \times \mathbb{Z}_{p^n})$ & $3^l = 2p^n + 1$  \\
\hline
$2$ & $(\mathbb{Z}_7 \times \mathbb{Z}_7) \rtimes (\GL_2 (3))^*$ & $(\GL_2 (3))^*$ is the group isoclinic to $\GL_2 (3)$ \\
& & that is a Frobenius complement \\
\hline
$2$ & $(\mathbb{Z}_5 \times \mathbb{Z}_5) \rtimes S$ & $S$ is the one of the three groups that acts \\
& & transitively on nonidentity elements of the elementary \\
& & abelian group of order $25$ and have two $2$-regular classes\\
\hline
$2$ & $(\mathbb{Z}_3)^4 \rtimes S$ & $S$ is one of the three groups that acts transitively \\
 & & on the nonidentity elements of the elementary abelian \\
 & & group of order $81$ and has two $2$-regular classes \\
\hline
\hline
\end{tabular}$$
\caption {Solvable groups with three $p$-regular classes not of the form $V \rtimes P$}\label{table 2c}

%\caption{Groups with two super-Brauer character theories}\label{table 2}
\end{table}
%%%---------------------------------------------------------------------------------------------

%\end{CJK*}
\end{document}